\documentclass[12pt,a4paper]{article}
\usepackage{amsmath}
\usepackage{amssymb}
\usepackage{enumitem}
\usepackage{tikz}
\usepackage[utf8]{inputenc}
\usepackage[english]{babel}
\usepackage{amsthm}
\usepackage{mathtools}
\usepackage{graphicx}

\advance\textwidth2\evensidemargin 
\title{Chromatic number is Ramsey distinguishing}
\date{9 March 2022}
\author{Michael Savery\footnote{Present address: Mathematical Institute, University of Oxford, Oxford OX2 6GG, UK}}
\newcommand{\Address}{{
  \bigskip
  \footnotesize
  \textsc{Department of Pure Mathematics and Mathematical Statistics, University of Cambridge, Wilberforce Road, Cambridge CB3 0WB, UK}\par\nopagebreak
  \textit{E-mail address}: \texttt{savery@maths.ox.ac.uk}
}}

\begin{document}
\newtheorem{theorem}{Theorem}
\newtheorem{lemma}[theorem]{Lemma}
\newtheorem{question}{Question}
\newtheorem*{claim}{Claim}
\theoremstyle{definition}
\newtheorem{def*}[theorem]{Definition}
\maketitle

\begin{abstract}
A graph $G$ is Ramsey for a graph $H$ if every colouring of the edges of $G$ in two colours contains a monochromatic copy of $H$. Two graphs $H_1$ and $H_2$ are Ramsey equivalent if any graph $G$ is Ramsey for $H_1$ if and only if it is Ramsey for $H_2$. A graph parameter $s$ is Ramsey distinguishing if $s(H_1)\neq s(H_2)$ implies that $H_1$ and $H_2$ are not Ramsey equivalent. In this paper we show that the chromatic number is a Ramsey distinguishing parameter. We also extend this to the multi-colour case and use a similar idea to find another graph parameter which is Ramsey distinguishing.
\end{abstract}

\section{Introduction}\label{intro}

A graph $G$ is Ramsey for a graph $H$ if every colouring of the edges of $G$ in two colours contains a monochromatic copy of $H$, and we denote this by $G\rightarrow H$. We write $G\not \rightarrow H$ if $G$ is not Ramsey for $H$, and we denote the set of graphs which are Ramsey for $H$ by $\mathcal{R}(H)$. Two graphs $H_1$ and $H_2$ are Ramsey equivalent if $\mathcal{R}(H_1)=\mathcal{R}(H_2)$. If $H_1$ and $H_2$ are Ramsey equivalent then we write $H_1\sim H_2$, and we write $H_1\not \sim H_2$ if they are not. It is clear that Ramsey equivalence is an equivalence relation. The notion of Ramsey equivalence was first introduced by Szab\'o, Zumstein, and Z\"urcher in \cite{szabo}.

One of the fundamental questions concerning Ramsey equivalence, first posed by Fox, Grinshpun, Liebenau, Person, and Szab\'o in \cite{fox} but also touched upon in \cite{szabo}, is whether there exist two non-isomorphic connected graphs which are Ramsey equivalent. Indeed, the only non-isomorphic pairs of graphs known to be Ramsey equivalent are of the form $(H_1,H_2)$ where $H_1$ is a clique on $n$ vertices and $H_2$ is a disjoint union of a clique on $n$ vertices and some graph with clique number less than $n$ (addressed in \cite{bloom},  \cite{fox}, and \cite{szabo}).

One approach that can be used to investigate Ramsey non-equivalence is to consider graph parameters $s$ for which $s(H_1)\neq s(H_2)$ implies that $H_1$ and $H_2$ are not Ramsey equivalent. Such parameters are called Ramsey distinguishing. Two important graph parameters known to be Ramsey distinguishing are the clique number $\omega$ and the odd girth $g_o$. Indeed, in \cite{nesetril} Ne\v{s}et\v{r}il and R\"{o}dl built on Folkman's work in \cite{folkman} to show that for every graph $H$ there exists a graph $G$ with $\omega(G)=\omega(H)$ such that $G\rightarrow H$. Then similarly in \cite{nesetril_2} they showed that for every graph $H$ there exists a graph $G$ with $g_o(G)=g_o(H)$ such that $G\rightarrow H$.

The only other graph parameter previously shown to be non-trivially Ramsey distinguishing (up to a small technicality) is the 2-density $m_2$, defined for a graph $G$ with at least three vertices by \[m_2(G) = \max_{\substack{H\subseteq G\\|V(H)|\geq 3}} \bigg\{\frac{|E(H)|-1}{|V(H)|-2}\bigg\},\] where $H\subseteq G$ means $H$ is a subgraph of $G$.

Indeed, let $G(n,p)$ be the random graph on $n$ vertices in which each edge is present independently with probability $p$. A theorem of R\"{o}dl and Ruci\'{n}ski proved in \cite{RR1,RR2,RR3} (see also \cite{nenadov} for a shorter proof due to Nenadov and Steger) states that for each fixed graph $F$ which is not a forest of stars and paths of length 3, there exist positive constants $c$ and $C$ such that \[\lim_{n\rightarrow\infty}\mathbb{P}\big(G(n,p) \rightarrow F\big) = \begin{cases} 0 &\text{if $p\leq cn^{-1/m_2(F)}$},\\ 1 &\text{if $p\geq Cn^{-1/m_2(F)}$}.
\end{cases}\] It is a straightforward consequence of this theorem that among graphs which are not forests of stars and paths of length 3, the 2-density is Ramsey distinguishing.

If a graph has no component of size at least 3, then clearly it can only be Ramsey equivalent to graphs which also have this property. Any forest with a component of size at least 3 has 2-density equal to 1, and it was shown by Axenovich, Rollin, and Ueckerdt in \cite{axenovich} that no forest is Ramsey equivalent to a graph containing a cycle. So in fact the 2-density is a Ramsey distinguishing parameter among graphs which have a component of size at least 3, and graphs which do not have such a component are not Ramsey equivalent to any which do.

In this paper we consider another graph parameter, the chromatic number, $\chi$. This is defined as the minimum number of colours required in a proper vertex colouring of the graph, that is, in a colouring of the vertices of the graph in which no two adjacent vertices receive the same colour.

The smallest chromatic number of a graph in $\mathcal{R}(H)$, called the chromatic Ramsey number of $H$, was determined for all $H$ by Burr, Erd\H{o}s, and Lov\'asz in \cite{burr}. It is not known whether there exists a pair of graphs with the same chromatic Ramsey number but different chromatic numbers, although it seems likely that such graphs do exist. Indeed, the chromatic Ramsey number of $K_4$, the complete graph on four vertices, is 18 (since we can 2-edge-colour any 17-vertex-colourable graph without a monochromatic $K_4$ by naturally extending from such a colouring of $K_{17}$), but it was shown in \cite{zhu} that there exists a graph of chromatic number five and chromatic Ramsey number 17. Regardless, this illustrates that the chromatic Ramsey number is unlikely to be a useful tool for determining whether or not the chromatic number is Ramsey distinguishing, unlike the corresponding quantities in the clique number and odd girth cases.

The Ramsey distinguishing properties of the chromatic number were investigated by Axenovich, Rollin, and Ueckerdt in \cite{axenovich}. They observe that if $G$ and $H$ are graphs with $G$ bipartite and $\chi(H)>2$ then $G\not\sim H$ since any sufficiently large complete bipartite graph is Ramsey for $G$ \cite{beineke} but does not contain $H$, and hence chromatic number is distinguishing for bipartite graphs.

They go on to establish that if $G$ and $H$ are graphs with $\chi(G)<\chi(H)$ and with $G$ \textit{clique-splittable} or \textit{odd-girth-splittable}, meaning that the vertex set of $G$ can be partitioned into two subsets each inducing a graph of smaller clique number than $G$ or greater odd girth than $G$ respectively, then $G\not \sim H$. We note that the odd-girth-splittability result in particular implies that chromatic number is Ramsey distinguishing for graphs of chromatic number three and four since such graphs are odd-girth-splittable.

The main result of this paper is that the chromatic number is a Ramsey distinguishing parameter.
\begin{theorem}\label{main}
Let $G$ and $H$ be graphs with $\chi(G)<\chi(H)$. Then there exists a graph $F$ such that $F\rightarrow G$ and $F\not \rightarrow H$.
\end{theorem}

In the next section we give some preliminary lemmas and then prove Theorem~\ref{main}. In the final section we explain how to modify the proof of the theorem to extend to the case of $q$-edge-colourings for all integers $q\geq 2$, and we use a similar proof to show that a related parameter (namely the minimum number of vertices which can be given colour 1 in a proper vertex colouring of $G$ in colours $1,\ldots,\chi(G)$) is also Ramsey distinguishing. We then close with a discussion of some interesting open problems in this area.

For a graph $G=(V,E)$ and a set of vertices $S\subseteq V$, we will write $G[S]$ to mean the graph with vertex set $S$ and edge set $\{uv\in E : u,v\in S\}$. We will write $G-S$ to mean $G[V\char`\\ S]$. We denote by $K_r$ the complete graph on $r$ vertices and write $K_r(n)$ for the complete $r$-partite graph with $n$ vertices in each partite class. For graphs $G$ and $H$ we denote by $G+H$ the graph consisting of a copy of $G$ and a copy of $H$ on disjoint vertex sets.

\section{Preliminary lemmas and proof of Theorem~\ref{main}}\label{sect_main_proof}

A $k$-uniform hypergraph is a hypergraph in which every hyperedge has size $k$. By a circuit of length $s$ in a hypergraph $\mathcal{H}=(V,\mathcal{E})$ we mean a sequence of distinct vertices $v_1,\ldots ,v_s\in V$, and a sequence of distinct hyperedges $e_1,\ldots,e_s\in \mathcal{E}$ such that $v_i,v_{i+1}\in e_i$ for $1\leq i<s$ and $v_s,v_1\in e_s$. Note that in particular we consider two distinct hyperedges intersecting in two vertices to give rise to a circuit of length two. By the girth of a hypergraph we mean the length of the shortest circuit in the hypergraph. The independence number of a hypergraph is the size of the largest set of its vertices which does not contain a hyperedge. 

The following lemma was proved by Erd\H{o}s and Hajnal in \cite{erdos}.
\begin{lemma}\label{hyps}
Let $k, N\geq 2$ be integers and let $\epsilon>0$. Then there exists a $k$-uniform hypergraph $\mathcal{H}=(V,\mathcal{E})$ with girth greater than $N$ and independence number less than $\epsilon|V|$.
\end{lemma}

We will use such hypergraphs to construct graphs with desirable properties. The next lemma relates the girth of a hypergraph to the properties of small cycles in certain graphs constructed from that hypergraph.

\begin{lemma}\label{cycles}
Let $N\geq 2$ be an integer and let $\mathcal{H}=(V,\mathcal{E})$ be a hypergraph of girth greater than $N$. Let $G$ be a graph with vertex set $V$ such that for distinct $u,v\in V$, $u$ and  $v$ are adjacent in $G$ only if $u$ and $v$ have a common hyperedge in $\mathcal{E}$. Suppose there exists a cycle $C=v_1v_2\ldots v_kv_1$ in $G$ of length $k\leq N$. Let $e\in \mathcal{E}$ be a hyperedge containing some edge of $C$. Then $e$ contains every vertex of $C$.
\end{lemma}
\begin{proof}
Suppose for a contradiction that $e$ does not contain every vertex of $C$. Without loss of generality, $v_1,v_2\in e$. Define hyperedges $e_i\in \mathcal{E}$ for $1\leq i\leq k$ by letting $e_1=e$, by letting $e_i$ contain $v_i$ and $v_{i+1}$ for $1<i<k$, and by letting $e_k$ contain $v_k$ and $v_1$. Since $e_1$ does not contain every vertex of $C$, we may assume that $e_1$ and $e_k$ are distinct by choosing a suitable labeling of the vertices and hyperedges.

Define $i_1=1$ then recursively define $i_j$ for $j\geq 2$ to be least such that $e_{i_j}$ and $e_{i_{j-1}}$ are distinct, if such exists. Let $J$ be the maximum $j$ for which $i_j$ is defined. Note that $e_{i_J}=e_k$.

There are two cases to consider. First, if the sequence $e_{i_1},\ldots,e_{i_J}$ contains a repeated entry, then let $a,b\in\{1,\ldots,J\}$ with $a<b$ be such that $e_{i_a}=e_{i_b}$ and the sequence $e_{i_{a+1}},e_{i_{a+2}},\ldots,e_{i_b}$ contains no repeated entry (note that $b\geq a+2$ so this list has at least two members). This sequence of hyperedges with vertices $v_{i_{a+1}},\ldots,v_{i_{b}}$ forms a circuit in $\mathcal{H}$ of length less than $N$, which is a contradiction.

In the other case, the sequence $e_{i_1},\ldots,e_{i_J}$ contains no repeated entry so with vertices $v_{i_1},\ldots,v_{i_J}$ forms a circuit in $\mathcal{H}$ of length at most $N$, which again is a contradiction.
\end{proof}

\begin{def*}\label{prop_p}
Let $N, k_1, k_2\in \mathbb{N}$. We say that a graph $G$ has property $P(N,k_1,k_2)$ if it has a red-blue edge colouring such that every red subgraph of $G$ on $N$ vertices or fewer is $k_1$-vertex-colourable and every blue subgraph on $N$ vertices or fewer is $k_2$-vertex-colourable. We call such a colouring of $G$ an $(N,k_1,k_2)$-good colouring.
\end{def*}

The next lemma is based on a construction of Fox, Grinshpun, Liebenau, Person, and Szab\'o in \cite{fox}.

\begin{lemma}\label{l_exist}
Let $G$ be a graph with at least two vertices, let $\epsilon>0$ and let $N\geq 2$ be an integer. Then there exists a graph $L(G,\epsilon,N)$ such that the subgraph of $L$ induced by any set of at least $\epsilon|V(L)|$ of its vertices contains $G$ as a subgraph, and moreover such that if $G$ has property $P(N,k_1,k_2)$ for some $k_1, k_2\in \mathbb{N}$ then $L$ has this property too.
\end{lemma}
\begin{proof}
Let $\mathcal{H}=(V,\mathcal{E})$ be a $|V(G)|$-uniform hypergraph with independence number less than $\epsilon|V|$ and girth greater than $N$. This exists by Lemma~\ref{hyps}. Construct a graph $L$ on vertex set $V$ by placing a copy of $G$ in each hyperedge in $\mathcal{E}$. Then since $\mathcal{H}$ has independence number less than $\epsilon|V|$, any set of at least $\epsilon|V|$ of its vertices contains some hyperedge, so the subgraph of $L$ induced by these vertices contains a copy of $G$.

Suppose that $G$ has property $P(N,k_1,k_2)$ for some $k_1, k_2\in \mathbb{N}$. Colour the edges of $L$ by colouring each of the copies of $G$ with an $(N,k_1,k_2)$-good colouring. We claim that this is an $(N,k_1,k_2)$-good colouring of $L$.

Let $H$ be a red subgraph of this colouring with $N$ vertices or fewer. Suppose for a contradiction that $H$ has chromatic number $t>k_1$, and assume that $H$ is minimal with this property. Note that $H$ contains an edge and is not contained in any of the copies of $G$. Let $G_0$ be a copy of $G$ in $L$ which contains an edge of $H$. Let $W_0$ be the vertex set of $G_0$. Let $H_0$ be a connected component of $H[W_0]$ containing at least one edge. Let $V_0$ be the vertex set of $H_0$. 

Consider the connected components of $H-V_0$, and label them $H_1,H_2,\ldots,H_m$ with vertex sets $V_1,V_2,\ldots,V_m$ respectively. By Lemma~\ref{cycles}, any cycle in $L$ of length at most $N$ and containing an edge of $H_0$ must be contained in $G_0$. Hence for every $1\leq i\leq m$, there exists at most one vertex in $V_0$ which has edges to $H_i$. If such a vertex exists, then label it $v_i$ and define ${H'_i}=H[V_i\cup \{v_i\}]$. Otherwise let $H'_i=H[V_i]$. Then either $H_0$ is $t$-chromatic, or $H'_i$ is $t$-chromatic for some $1\leq i\leq m$. All of these graphs have fewer vertices than $H$ which contradicts the minimality of $H$.

Therefore $H$ is $k_1$-colourable, and similarly any blue subgraph of this colouring with $N$ vertices or fewer is $k_2$-colourable.
\end{proof}

Before proving Theorem~\ref{main} we finally state the following `focussing lemma' proved in \cite{fox}.
\begin{lemma}\label{focussing}
Let $G$ be a complete bipartite graph with partite sets $A$ and $B$. Consider a 2-edge-colouring of the edges of $G$. Then there exist subsets $A'\subseteq A$ and $B'\subseteq B$ with $|A'|\geq |A|/2$ and $|B'|\geq |B|/2^{|A|}$ such that the complete bipartite graph with partite sets $A'$ and $B'$ is monochromatic.
\end{lemma}

We're now ready to prove Theorem~\ref{main}.

\begin{proof}[Proof of Theorem~\ref{main}]
If $\chi(G)=1$ then the result is trivial so assume otherwise. Let $M = |V(G)|$, let $N = |V(H)|$, and let $\chi=\chi(G)$. Note that $M,N\geq 2$. We will construct a graph which is Ramsey for $K_\chi(M)$ (and hence for $G$), but which has property $P(N,\chi,\chi)$ (and hence has a colouring without a monochromatic copy of $H$).

We recursively define graphs $F_i$ for $i\in \mathbb{N}_0$. Let $F_0$ be the graph consisting of $M$ independent vertices. For $i\in \mathbb{N}_0$, let $L_i$ be the graph $L(F_i,\epsilon,N)$ given by Lemma~\ref{l_exist}, where $\epsilon = 2^{-2M}$. Then define $F_{i+1}$ to consist of a copy of $L_i$ on vertex set $B_{i+1}$ and a disjoint set $A_{i+1}$ of $2M$ independent vertices, with $A_{i+1}$ and $B_{i+1}$ forming the partite sets of a complete bipartite graph (see Figure 1).

\begin{figure}[ht]\label{main_fig}
	\centering
	\includegraphics[width=0.8\linewidth]{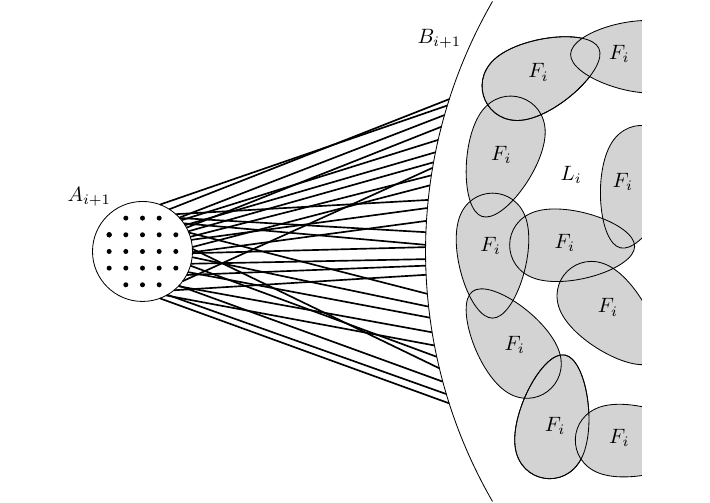}
	\caption{The construction of $F_{i+1}$ from $F_i$.}
\end{figure}

\begin{claim}
For all $i\in \mathbb{N}_0$ and for all $k_1,k_2\in \mathbb{N}$ such that $k_1+k_2=i+2$, $F_i$ has property $P(N,k_1,k_2)$.
Moreover, for all $i\in \mathbb{N}_0$ and for any 2-edge-colouring of $F_i$ there exist $m_1,m_2\in \mathbb{N}$ with $m_1+m_2=i+2$ such that $F_i$ contains monochromatic copies of $K_{m_1}(M)$ and $K_{m_2}(M)$ in different colours.
\end{claim}

We prove both parts of the claim by induction on $i$. For $i=0$, trivially $F_0$ has property $P(N,1,1)$ and any 2-edge-colouring of $F_0$ contains both a red and blue copy of $K_1(M)$. Now let $i\in\mathbb{N}_0$ and consider $F_{i+1}$. Let $k_1,k_2\in \mathbb{N}$ with $k_1+k_2=i+3$. Without loss of generality $k_1\neq 1$, so $F_i$ has property $P(N,k_1-1,k_2)$ by the induction hypothesis, and hence by construction $L_i$ has this property too. Consider a colouring of $F_{i+1}$ in which we colour the copy of $L_i$ with an $(N,k_1-1,k_2)$-good colouring and colour the complete bipartite graph joining $A_{i+1}$ and $B_{i+1}$ red. This is an $(N,k_1,k_2)$-good colouring of $F_{i+1}$ as required.

Next consider a 2-edge-colouring of $F_{i+1}$. By Lemma~\ref{focussing} there exist subsets $A'_{i+1}\subseteq A_{i+1}$ and $B'_{i+1}\subseteq B_{i+1}$ with $|A'_{i+1}|=|A_{i+1}|/2=M$ and $|B'_{i+1}|\geq|B_{i+1}|/2^{2M} = \epsilon|B_{i+1}|$ such that the complete bipartite graph between $A'_{i+1}$ and $B'_{i+1}$ is monochromatic. By the construction of $L_i$, the subgraph of $L_i$ induced by $B'_{i+1}$ contains a copy of $F_i$. Hence the 2-edge-colouring of $F_{i+1}$ contains a copy of $F_i$ and a disjoint set of $M$ independent vertices which form the partite sets of a monochromatic complete bipartite graph.

By the induction hypothesis there exist $m'_1,m'_2\in \mathbb{N}$ with $m'_1+m'_2=i+2$ such that this copy of $F_i$ contains monochromatic copies of $K_{m'_1}(M)$ and $K_{m'_2}(M)$ in different colours. With the monochromatic complete bipartite graph from the copy of $F_i$ to the disjoint set of $M$ vertices, these either form monochromatic copies of $K_{m'_1+1}(M)$ and $K_{m'_2}(M)$ in different colours, or monochromatic copies of $K_{m'_1}(M)$ and $K_{m'_2+1}(M)$ in different colours. This completes the proof of the claim.

Finally, consider $F_{2\chi-3}$. By the claim this has property $P(N,\chi,\chi-1)$ (and hence in particular property $P(N,\chi,\chi)$), and any 2-edge-colouring of it contains a monochromatic $K_\chi(M)$. It is therefore Ramsey for $G$ but not for $H$, as required.
\end{proof}

\section{Extensions and open problems}\label{sect_ext_open}

In this section we will give two further results which can be proved using a similar technique to that used on Theorem~\ref{main}, and then highlight some interesting related open problems.

\subsection{Generalisation to multiple colours}\label{sub_multi}
We first consider the natural generalisation of the notion of Ramsey equivalence to that of $q$-Ramsey equivalence for an integer $q\geq 2$. A graph $G$ is called $q$-Ramsey for $H$ if every $q$-colouring of the edges of $G$ contains a monochromatic copy of $H$, and we write this as $G\rightarrow (H)_q$. We say that $H_1$ and $H_2$ are $q$-Ramsey equivalent, or just $q$-equivalent, if for all graphs $G$, $G\rightarrow (H_1)_q$ if and only if $G\rightarrow (H_2)_q$. The introduction of \cite{clemens} provides a good summary of what is known about $q$-equivalence, and about the interplay between $q$-equivalence and $r$-equivalence for $q\neq r$. 

In particular, in that paper it is shown that the graphs $K_3$ and $K_3 + K_2$ are $q$-equivalent for all $q\geq 3$, but it is straightforward to see that $K_6$ is $2$-Ramsey for $K_3$ but not for $K_3+K_2$ (e.g. see \cite{szabo}). This implies that in general we cannot deduce anything about $q$-equivalence from non-2-equivalence for $q\geq 3$. Hence the following theorem is not an immediate consequence of Theorem~\ref{main}.

\begin{theorem}\label{multi}
Let $G$ and $H$ be graphs with $\chi(G)<\chi(H)$, then for all integers $q\geq 2$ there exists a graph $F$ such that $F\rightarrow (G)_q$ and $F\not \rightarrow (H)_q$.
\end{theorem}

We can prove this result by modifying the arguments of Section~\ref{sect_main_proof}. For given $q\geq 2$ and $N,k_1,\ldots,k_q\in \mathbb{N}$, we will say that graph $G$ has property $P_q(N,k_1,\ldots,k_q)$ if it has an edge colouring in colours $1,\ldots,q$ such that every monochromatic subgraph of $G$ in colour $i$ on $N$ vertices or fewer is $k_i$-vertex-colourable. Call such a colouring an $(N,k_1,\ldots,k_q)_q$-good colouring of $G$. We can then generalise Lemma~\ref{l_exist} by replacing `$P(N,k_1,k_2)$ for some $k_1,k_2\in \mathbb{N}$' in its statement by `$P_q(N,k_1,\ldots,k_q)$ for some $q\geq 2$ and $k_1,\ldots,k_q\in \mathbb{N}$'. This can be proved by a straightforward modification of the proof of Lemma~\ref{l_exist} given above (indeed, the construction of the graph $L$ need not be changed).

A straightforward generalisation of the proof of the focussing lemma in \cite{fox} gives the following: \textit{let $G$ be a complete bipartite graph with partite sets $A$ and $B$. Let $q\geq 2$ be an integer and consider a $q$-edge-colouring of the edges of $G$. Then there exist subsets $A'\subseteq A$ and $B'\subseteq B$ with $|A'|\geq |A|/q$ and $|B'|\geq |B|/q^{|A|}$ such that the complete bipartite graph with partite sets $A'$ and $B'$ is monochromatic.}

Finally we modify the proof of Theorem~\ref{main}. As before, if $\chi(G)=1$ then the result is trivial so assume otherwise. Define $M=|V(G)|$, $N=|V(H)|$, and $\chi=\chi(G)$, noting that $M,N\geq2$. Define $F_0$ as before, then for $i\in\mathbb{N}_0$ define $L_i$ to be the graph $L(F_i,\epsilon,N)$ given by the modified Lemma~\ref{l_exist}, where $\epsilon=q^{-qM}$. Define $F_{i+1}$ to consist of a copy of $L_i$ on vertex set $B_{i+1}$ and a disjoint set $A_{i+1}$ of $qM$ independent vertices with these two forming the partite sets of a complete bipartite graph.

We make the following modified claim: \textit{for all $i\in \mathbb{N}_0$ and for all $k_1,\ldots,k_q\in\mathbb{N}$ such that $k_1+\ldots+k_q=i+q$, $F_i$ has property $P_q(N,k_1,\ldots,k_q)$. Moreover, for all $i\in \mathbb{N}_0$ and for any $q$-edge-colouring of $F_i$ there exist $m_1,\ldots,m_q\in \mathbb{N}$ with $m_1+\ldots+m_q=i+q$ such that $F_i$ contains monochromatic copies of $K_{m_1}(M)$,\ldots,$K_{m_q}(M)$ in colours $1,\ldots,q$ respectively.}

This can be proved with a straightforward generalisation of the proof of the original claim. We finally consider $F_{q(\chi-2)+1}$, which by the claim has property $P_q(N,\chi,\chi-1,\ldots,\chi-1)$ so is not $q$-Ramsey for $H$, but any $q$-edge-colouring of it contains a monochromatic $K_\chi(M)$ so it's $q$-Ramsey for $G$.

We remark that this construction can also be applied to the asymmetric version of the problem. For $q\geq 2$, we say that a graph $G$ is Ramsey for a $q$-tuple of graphs $(H_1,\dots,H_q)$ if for every colouring of the edges of $G$ in colours $1,\dots,q$, there exists $i\in\{1,\dots,q\}$ such that $G$ contains a monochromatic copy of $H_i$ in colour $i$. We will say that two $q$-tuples of graphs are Ramsey equivalent if the sets of graphs which are Ramsey for each are equal.

Let $(H_1,\dots,H_q)$ and $(H_1',\dots,H_q')$ be two $q$-tuples of graphs. Let $\chi_i=\chi(H_i)$ and $\chi_i'=\chi(H_i')$ for all $i$. Suppose that $\chi_i, \chi_i'\geq 2$ for all $i$, and that $\sum_{i=1}^q\chi_i<\sum_{i=1}^q\chi_i'$. Setting $M=\max\{|V(H_i)|:1\leq i \leq q\}$ and $N=\max\{|V(H_i')|:1\leq i \leq q\}$, then defining graphs $F_i$ as above, we see that by the modified claim $F_{\sum_i\chi_i-2q+1}$ has property $P_q(N, \chi_1'-1,\dots,\chi_q'-1)$ and hence is not Ramsey for $(H_1',\dots,H_q')$. Also by the modified claim, $F_{\sum_i\chi_i-2q+1}$ contains an $i$-coloured copy of $K_{\chi_i}(M)$ for some $i$ and hence is Ramsey for $(H_1,\dots,H_q)$. Thus if two $q$-tuples of graphs, each of which have chromatic number at least 2, have different sums of chromatic numbers, then they are not Ramsey equivalent.

\subsection{Another Ramsey distinguishing parameter}\label{sub_a}

For an integer $q\geq 2$ we describe a graph parameter $s$ as being $q$-Ramsey distinguishing, or just $q$-distinguishing, if $s(H_1)\neq s(H_2)$ implies that $H_1$ and $H_2$ are not $q$-equivalent. In \cite{nesetril, nesetril_2} Ne\v{s}et\v{r}il and R\"{o}dl showed that for all $q\geq2$ and for all graphs $H$ there exist graphs $G_1$ and $G_2$ with $\omega(G_1)=\omega(H)$ and $g_o(G_2)=g_o(H)$ such that $G_1\rightarrow (H)_q$ and $G_2\rightarrow (H)_q$. Hence clique number and odd girth are $q$-distinguishing for all $q\geq2$. 

R\"{o}dl and Ruci\'{n}ski's random Ramsey theorem described in Section \ref{intro} for 2 colours generalises to more colours (again see \cite{RR1, RR2, RR3} for their original proof, and \cite{nenadov} for a shorter, more recent proof due to Nenadov and Steger). The multi-colour result states that for each fixed $q>2$ and graph $F$ which is not a forest of stars, there exist positive constants $c$ and $C$ such that \[\lim_{n\rightarrow\infty}\mathbb{P}\big(G(n,p) \rightarrow (F)_q\big) = \begin{cases} 0 &\text{if $p\leq cn^{-1/m_2(F)}$},\\ 1 &\text{if $p\geq Cn^{-1/m_2(F)}$},
\end{cases}\] where we recall that $m_2(F)$ denotes the 2-density of $F$.

If $F$ is a forest of stars and $q>2$, then if $G$ is a disjoint union of sufficiently many sufficiently large stars, then $G$ is $q$-Ramsey for $F$, but not $q$-Ramsey for any graph which is not a forest of stars. Thus a forest of stars can only be $q$-equivalent to another forest of stars. All forests with a component of size at least 3 have 2-density equal to 1, and if a graph has no component of size at least 3, then it can only be $q$-equivalent to other graphs of this type. Thus for all $q>2$ the 2-density is $q$-distinguishing among graphs with a component of size at least 3, and no graph with such a component is $q$-equivalent to one without. Recall from Section \ref{intro} that the same is true for the case $q=2$.

In Section~\ref{sub_multi} we showed that chromatic number is $q$-distinguishing for all $q\geq 2$. Define $a(G)$ to be the minimum number of vertices which can be given colour 1 in a proper vertex colouring of $G$ in colours $1,\ldots,\chi(G)$. We can use a similar proof to show that $a$ is $q$-distinguishing for all $q\geq 2$.

\begin{theorem}\label{extra_distinguishing}
Let $G$ and $H$ be graphs with $a(G)<a(H)$, then for all integers $q\geq 2$, $G$ and $H$ are not $q$-Ramsey equivalent.
\end{theorem}

Indeed, let $q\geq 2$ be an integer and note that by Theorem~\ref{multi} we can assume that $\chi(G)=\chi(H)$. If $\chi(G)=1$ then the result is trivial, so assume otherwise and define $M = |V(G)|$, $N = |V(H)|$, $\chi = \chi(G)$ and $\epsilon=q^{-qM}$, then construct graphs $F_i$ for $i\in\mathbb{N}_0$ as in Section~\ref{sub_multi}.

Consider $F_{q(\chi-2)}$. We know that this has property $P_q(N,\chi-1,\ldots,\chi-1)$ and every $q$-edge-colouring of it either contains a monochromatic $K_\chi(M)$ or monochromatic copies of $K_{\chi-1}(M)$ in every colour. Let $L=L(F_{q(\chi-2)},q^{-q\cdot a(G)},N)$ from the modified Lemma~\ref{l_exist} from Section~\ref{sub_multi}, then define $F$ to consist of a copy of $L$ on vertex set $B$ and a disjoint set $A$ of $q\cdot a(G)$ independent vertices, with $A$ and $B$ forming the partite sets of a complete bipartite graph.

Then $F\rightarrow (G)_q$ because by the modified focussing lemma from Section~\ref{sub_multi}, any $q$-edge-colouring of $F$ contains a copy of $F_{q(\chi-2)}$ and a disjoint set of $a(G)$ independent vertices which form the partite sets of a monochromatic complete bipartite graph. Either this copy of $F_{q(\chi-2)}$ contains a monochromatic copy of $K_\chi(M)$ (and hence a monochromatic copy of $G$) or it contains monochromatic copies of $K_{\chi-1}(M)$ in every colour, which with the monochromatic complete bipartite graph to the set of $a(G)$ independent vertices gives a monochromatic copy of $G$.

We now show $F\not \rightarrow (H)_q$ by giving a $q$-edge-colouring of $F$ not containing a monochromatic copy of $H$. Colour the copy of $L$ in $F$ with an $(N,\chi-1,\ldots,\chi-1)_q$-good colouring (such a colouring exists by the construction of $L$). Partition $A$ into $q$ sets $S_1,\ldots,S_q$ each containing $a(G)$ vertices, then colour the edges of $F$ containing a vertex in $S_i$ with colour $i$ for each $1\leq i\leq q$. Suppose this colouring contains some monochromatic copy of $H$, say on vertex set $V$. Colour the vertices of this copy of $H$ by colouring the vertices in $V\cap B$ with a proper $(\chi-1)$-colouring in colours $2,\ldots,\chi$ and colouring the vertices in $V\cap A$ with colour 1. This gives a proper $\chi$-vertex-colouring of $H$ in which at most $a(G)<a(H)$ vertices receive colour 1, which is a contradiction.

\subsection{Related open problems}\label{sub_open}
In this paper we have found two new Ramsey distinguishing parameters. It is natural to start with the following question, asked by Axenovich, Rollin, and Ueckerdt in \cite{axenovich}.
\begin{question}\label{q_others}
Do there exist other graph parameters which are Ramsey distinguishing?
\end{question}

As mentioned in the introduction, the next question is one of the most fundamental concerning Ramsey equivalence. It was first posed in \cite{fox}.
\begin{question}\label{q_main}
Do there exist two non-isomorphic connected graphs which are Ramsey equivalent?
\end{question}
The results in this paper add to a body of evidence for a negative answer to Question~\ref{q_main}. Using Ramsey distinguishing parameters we know that any pair of graphs $H_1$ and $H_2$ that are Ramsey equivalent must have the same clique number, chromatic number, odd girth, minimal partite set size, and, if they both have a component of size at least 3, the same 2-density. A graph $G$ is called Ramsey isolated if $G\not\sim H$ for all connected graphs $H$ which are not isomorphic to $G$. It is known that every clique \cite{fox}, every path, every star, and every connected graph on at most five vertices (all \cite{axenovich}) is Ramsey isolated. Furthermore, it was shown by Clemens, Liebenau, and Reding in \cite{clemens} that no pair of non-isomorphic 3-connected graphs are Ramsey equivalent.

In light of Theorem~\ref{main} we can identify some new families of Ramsey isolated graphs. In \cite{axenovich} it was shown that if a connected graph $G$ satisfies \textit{(i)} there exists an independent set $S\subseteq V(G)$ such that $\omega(G-S)<\omega(G)$, and \textit{(ii)} there exists a proper $\chi(G)$-vertex-colouring of $G$ in which the subgraph induced on some two colour classes is a matching, then any connected graph $H$ which is Ramsey equivalent to $G$ but is not isomorphic to $G$ satisfies $\chi(H)<\chi(G)$. They also observe that the same is true if we replace the clique number by the negative of odd girth in \textit{(i)}. Call this modified version of the condition \textit{(i)'}. 

Combining these results with Theorem~\ref{main} we have that any connected graph $G$ satisfying \textit{(ii)} and either \textit{(i)} or \textit{(i)'} is Ramsey isolated. In particular this implies that every odd cycle is Ramsey isolated.

Finally, Theorems~\ref{multi} and \ref{extra_distinguishing} support positive answers to the following two questions posed in \cite{clemens} concerning the interplay between 2- and 3-equivalence.
\begin{question}\label{q_multi}
\begin{enumerate}[label=(\alph*)]
\item\label{q_multi_a} If graphs $H_1$ and $H_2$ are 2-equivalent, must they also be 3-equivalent?
\item If connected graphs $H_1$ and $H_2$ are 3-equivalent, must they also be 2-equivalent?
\end{enumerate}
\end{question}
It was shown in \cite{clemens} that for non-negative integers $a,b,q,$ and $r$ where $q,r\geq 2$, if $H_1$ and $H_2$ are $q$- and $r$-equivalent then they are $(aq+br)$-equivalent. Thus a positive answer to Question~\ref{q_multi}\ref{q_multi_a} would imply that if two graphs are 2-equivalent, then they are $q$-equivalent for all $q\geq 2$.

\section*{Acknowledgements}\label{ack}
The author would like to thank the University of Cambridge Faculty of Mathematics Summer Research Scheme and the Lister Fund of Queens' College, Cambridge for funding this work, Julia Gog and James Kelly for their work to ensure funding was found, and Thomas Bloom for his guidance and helpful discussions.

Thanks also to the anonymous referees for their helpful comments and suggestions, particularly in pointing out that 2-density is Ramsey distinguishing, and that the construction in Section \ref{sub_multi} can be applied to the asymmetric version of the problem.

\Address

\end{document}